\documentclass{article}
\usepackage{graphicx}
\usepackage{amssymb}
\usepackage{amsmath}
\usepackage{a4wide}
\usepackage{caption}
\usepackage{subcaption}
\usepackage[affil-it]{authblk}
\usepackage{xr}
\usepackage{color}
\usepackage{soul}
\usepackage{mathtools}

\usepackage[utf8]{inputenc} 
\usepackage[T1]{fontenc}
\newtheorem{theorem}{Theorem}[section]

\newtheorem{corollary}[theorem]{Corollary}

\newtheorem{definition}[theorem]{Definition}

\newtheorem{lemma}[theorem]{Lemma}

\newtheorem{proposition}[theorem]{Proposition}
\newtheorem{remark}[theorem]{Remark}

\newenvironment{proof}[1][Proof]{\textbf{#1.} }{\ \rule{0.5em}{0.5em}}

\begin{document}
\title{A note on generalization of 
Zermelo navigation problem \\ on Riemannian manifolds with strong perturbation}
\author{Piotr Kopacz}
\affil{\small{Jagiellonian University, Faculty of Mathematics and Computer Science\\ 6, Prof. S. Łojasiewicza, 30 - 348, Kraków,
Poland}}
\affil{Gdynia Maritime University, Faculty of Navigation \\3,  Al. Jana Pawła II, 81-345, Gdynia, Poland}
\date{\texttt{}}
\date{\normalsize{\texttt{piotr.kopacz@im.uj.edu.pl}}}
\maketitle
\begin{abstract}
\noindent
We generalize the Zermelo navigation problem and its solution on Riemannian manifolds $(M, h)$  admitting a space dependence of a ship's speed $0<|u(x)|_h\leq1$ 
in the presence of a perturbation $\tilde{W}$ determined by a strong velocity vector field satisfying $|\tilde{W}(x)|_h=|u(x)|_h$, with application of Finsler metric of Kropina type. 
\end{abstract}

\ 

\smallskip
\noindent
\textbf{M.S.C. 2010}: 53B20, 53C21, 53C22, 53C60, 49J15, 49J53.



\section{Introduction}

The objective in the navigation problem of Zermelo is to find the minimum time paths of a ship sailing on  a sea $M$, with the presence of a wind determined by a vector field $W$. The problem was formalized and investigated by E. Zermelo (1931) in the Euclidean spaces  $\mathbb{R}^2$ and $\mathbb{R}^3$, cf. \cite{zermelo2, zermelo}, and generalized considerably (2004) in \cite{colleen_shen} for the case when sea is a Riemannian manifold $(M, h)$ under the assumption that a wind $W$ is a time-independent weak wind, i.e. $h(W, W)<1$. In the absence of a perturbation the solutions to the problem are simply $h$-geodesics of $M$. Note that the original solution given by E. Zermelo admitted a strong wind. The problem may be considered as purely geometric. It has been found out that the trajectory which minimizes travel time are exactly the geodesics of a special Finsler type $F$, that is Randers metric. In other words, the solutions to the problem are the flows of Randers geodesics. The condition on strong convexity, i.e. $|W|_h<1$ ensures then that $F$ is a positive definite Finsler metric. Furthermore, there is an equivalence between Randers metrics and Zermelo's problems \cite{colleen_shen, chern_shen}.   

In \cite{kropina} the authors showed that Zermelo's navigation problem has another solution in Finsler geometry in the case when the wind becomes stronger. This means that there is a wind acting of about the same force as a maximal power of ship's engine. Precisely, it was assumed that $h(W, W)=1$. 
The problem was considered in the original formulation when a ship sails with Riemannian unit speed, i.e. $h(u, u)=1=const.$ 
Obviously, since the ship's speed $|u|_h$ and the wind force $|W|_h$ are equal, unlike the Randers case, the ship cannot proceed anymore against the wind. So following the direction $u=-W$ implies that the resultant velocity $v$ vanishes. Geometrically, in each tangent space $T_xM$ the unit sphere of the new metric $F$ is the $W$-translate of the Riemannian $h$-unit sphere. However, differently from the Randers case, the former passes through the origin of $T_xM$ and thus $F$ cannot be a Finlser metric in the classic sense \cite{kropina}.  

Setting as a reference point Zermelo's formulation of the problem we may ask whether a ship must proceed at a constant maximum speed relative to the surrounding Riemannian sea, i.e. $|u|_h=1$. This assumption we have already dropped  considering the problem on Riemannian manifolds, however being in the case of a background weak wind which guarantees a full control of navigating ship (cf. \cite{kopi6}). Reviewing the bibliography in this scope one may find the paper by A. de Mira Fernandes \cite{mira} who accomodated shortly after Zermelo's contribution a varying magnitude ship's velocity. Having added the extra degree of freedom the author allowed a time and space dependent velocity and solved the corresponding problem with the Euclidean background, namely in $\mathbb{R}^n$. Therefore, he has generalized the results of E. Zermelo \cite{zermelo2, zermelo} and T. Levi-Civita \cite{levi} for the Euclidean spaces to the case where the air speed of a plane is a preassigned function of position and time. Also, the subsequent equations for the flight path of least time  obtained by K. Arrow \cite{arrow}, who considered a passage with $\mathbb{S}^2$-background, implied earlier results achieved by T. Levi-Civita. De Mira Fernandes showed that the change in $|u|_h$ with time has no effect on the formula for the shortest time passage (time-optimal ship's heading) while that with space has the same effect as a corresponding change in wind \cite{arrow}. 
The above contribution was also referred in the modern approach \cite{herdeiro} in a discussion how, when both $W$ and $|u|_h$ are space but not time dependent, it can be recast in a purely geometric form as geodesics of a Randers geometry or as null geodesics in a stationary space-time. Let us note that the Zermelo navigation as a method plays an active and crucial role in modern physics, in particular in quantum mechanics. In this regards, see, for instance, the expositions in \cite{sanchez, russell, russell2, brody, brody2, brody3}. In the investigation on Riemannian manifolds for the case of a strong wind we are going to drop the standard assumption on a constant unit speed. 
We aim to present our glance at the problem with different starting point and therefore contribute to the earlier findings introducing a priori fixed space dependence of a ship's speed $0<|u(x)|_h\leq1$.  
 

\section{Glance at previous findings from a different perspective 
}

Let a pair $ (M,h) $ be a Riemannian manifold where $h = h_{ij}dx^i\otimes dx^j$ is a Riemannian metric and the corresponding norm-squared of tangent vectors $y \in T_x M$ is denoted by $\left|y \right|_h^2 = h_{ij}y^iy^j = h(y, y)$. In contrast to \cite{kropina} we begin with a Riemannian manifold $(M, h)$ and a vector field $\tilde{W}=\tilde{W}^i\frac{\partial}{\partial x}$ on $M$ which need not be of $h$-unit length. We admit that both ship's speed $|u(x)|_h$ and wind $W$ are space-dependent with $0<|u(x)|_h=|\tilde{W(}x)|_h\leq1$. Thus, a ship makes a way unceasingly through the water, but not necessarily over ground. We compute the new Finsler metric $\tilde{F}$ similarly as treated in \cite{kropina}, with a slight refinement of the initial indicatrix-based equation. To reduce the clutter we also adopt the same notations if not otherwise stipulated. We obtain the metric $\tilde{F}$ as the solution to the new equation including the new variable, that is 
\begin{equation}
\left|\frac{y}{\tilde{F}(x, y)}-\tilde{W}\right|=|u(x)|.
\end{equation}
It thus follows from the definition of the inner product 
\begin{equation}
h_{ij}(y^i-\tilde{F}\tilde{W}^i)(y^j-\tilde{F}\tilde{W}^j)=|u|^2\tilde{F}^2. 
\end{equation}
Hence, 
\begin{equation}
(|u|^2-|\tilde{W}|^2)\tilde{F}^2+2h(y, \tilde{W})\tilde{F}-|y|^2=0.
\end{equation}

\noindent
By assumption 
on the equality of the norms we are thus led to 
\begin{equation}
\tilde{F}(x, y)=\frac{|y|_h^2}{2h(y, \tilde{W}(x))}
\label{kropinka}
\end{equation}

\noindent
From the above concerned assumption it is implied that $\tilde{W}\neq0$; $y\neq 0$. We obtained the metric of the analogous form as $F$ in the original case of $h(\tilde{W}, \tilde{W})=1$. We also require that on $M$ there must exist a vector field $\tilde{W}$ without zeros. Therefore, having in mind the Poincar\'{e}-Hopf theorem 
one restricts the structures $(M, h)$ which the theory under consideration can be applied to. In particular, we exlude $\mathbb{S}^2$ since it follows that for any compact regular $2$-dimensional manifold with non-zero Euler characteristic, any continuous tangent vector field has at least one zero. 

\begin{remark}
Under a strong perturbation $|\tilde{W}|_h=1$ formula \eqref{kropinka} as a special case leads to the metric $F$ according to \cite{kropina} in the original formulation of the Zermelo navigation problem on Riemannian manifolds, i.e. with $h(u, u)=1.$ 
\end{remark}
\noindent
Let us observe that 
\begin{equation}
\tilde{F}(x, y)
=\frac{|y|_h^2}{2h(y, |u(x)|_hW(x))}
=\frac{1}{|u(x)|_h}F(x, y)
\label{conformality}
\end{equation}
where  $F(x, y)=\frac{|y|_h^2}{2h(y, W(x))}$ in the previous expression. From \eqref{conformality} it implies that $\tilde{F}$ is the conformal Finsler metric to $F$. This recalls the scenario in the generalized Randers case in the absence of a wind. Then, however, the resulting Randers metric is Riemannian and conformal to the corresponding background Riemannian metric $h$; see Proposition 2.5 in \cite{kopi6}. 
The adequate and wider investigation on conformal and weakly conformal Finsler geometry can be found, in particular, in \cite{rafie, matveev}. To proceed we can simply assume that 
\begin{equation}
\tilde{a}_{ij}(x)=h_{ij}, \qquad \tilde{b}_i(x)=2\tilde{W}_i.   
\label{ab0}
\end{equation}
Hence, 
\begin{equation}
\tilde{b}^2=\tilde{a}^{ij}\tilde{b}_i\tilde{b}_j=4|u(x)|_h^2
\end{equation}
while in the original setting one would then get $b^2=a^{ij}b_ib_j=4=const$. In order to avoid a constant function here as the obtained metrics could be a subject to such a constraint, the authors applied a conformal factor $e^{-k(x)}$ making use of some smooth function $k(x)$ on $M$. 
For comparison, to be in line with \cite{kropina}, if we use an analogous conformal factor $e^{-\tilde{k}(x)}$, where $\tilde{k}(x)$ is also some smooth function on $M$, then we get $\tilde{F}(x, y)=\frac{h(y, y)}{2h_{ij}\tilde{W}^jy^i}=\frac{e^{-\tilde{k}(x)}h_{ij}y^iy^j}{2e^{-\tilde{k}(x)}\tilde{W}_iy^i}$. Therefore, taking 
\begin{equation}
\tilde{a}_{ij}(x)=e^{-\tilde{k}(x)}h_{ij}, \qquad \tilde{b}_i(x)=2e^{-\tilde{k}(x)}\tilde{W}_i,  
\label{ab}
\end{equation}
yields the special Finsler type metric, namely the Kropina metric
\begin{equation}
\tilde{F}(x, y)=\frac{\tilde{a}_{ij}(x)y^iy^j}{\tilde{b}_i(x)y^i}=\frac{\tilde{\alpha}^2(x, y)}{\tilde{\beta}(x, y)}.
\label{kropinka2}
\end{equation}
$\tilde{F}(x, y)$ is composed of the new Riemannian metric $\tilde{\alpha}=\sqrt{\tilde{a}_{ij}(x)y^iy^j}$ and a $1$-form $\tilde{\beta}=\tilde{b}_i(x)y^i$ where $\tilde{b}^2=4|u(x)|_h^2e^{-\tilde{k}(x)}$. Conversely, if we put $h_{ij}=e^{\tilde{k}(x)}\tilde{a}_{ij}$ and $\tilde{W}_i(x)=\frac{e^{\tilde{k}(x)}\tilde{b}_i}{2}$, where 
\begin{equation}
\tilde{k}(x)=2\ln\left(\frac{2|u(x)|_h}{\tilde{b}(x)}\right)
\end{equation} 
then we obtain the initial navigation data in terms $h$ and $\tilde{W}$ which solution to the problem is exactly the Kropina metric \eqref{kropinka2}. To compare, recall $k(x)=\ln\frac{4}{b^2(x)}$ in the original setting. Note that it is sufficient to apply \eqref{ab0} in order to obtain the same form of the Kropina metric given by \eqref{kropinka2}. Therefore, $\tilde{b}^2\neq const.$ with $h(u, u)\neq const.$ Fulfilling the definition of Finsler metric which is positive definite, the function \eqref{kropinka2} is not defined on all $TM$, but only on a domain $\{(x, y)\in TM: \tilde{\beta}>0\}$.  Therefore, we exclude the case when $u=-\tilde{W}$. Following \cite{kropina} we have 
\begin{definition}
Let $(M, h)$ be an $n$-dimensional Riemannian space, $\tilde{W}$ a vector field globally defined on $M$. Let $\tilde{a}_{ij}$ and $\tilde{b}_{i}$ be given by \eqref{ab} 
and denote the Kropina metric by $\tilde{F}$, where $\tilde{F}=\frac{\tilde{\alpha}^2}{\tilde{\beta}}$. Then $\tilde{F}$ will be called $\tilde{U}$-Kropina metric.  
\end{definition}
\label{secondo}
Recall that since the Kropina metrics defined globally on $M$ are considered, the above mentioned topological restrictions to their existence occur. For more details see Propositions 5.2 and  5.13 in \cite{kropina}. One sees immediately that in the case of $h(\tilde{W}, \tilde{W})=1=const.$  $\tilde{U}$-Kropina metric becomes $U$-Kropina metric defined in the original presentation, that is Kropina metric with unit vector field. 
Recalling \eqref{ab0} and \eqref{kropinka2} it results that the generalization preserves the original Riemannian metric $\alpha$ but changes the $1$-form $\beta$. Comparing the resulting Finsler metrics 
we observe that $\tilde{\alpha}=\alpha$ and $\tilde{\beta}\neq\beta$ since $\tilde{W_i}\neq W_i$ for $h(u, u)\neq1$. The difference is made by perturbing wind what, in other words, is connected to the fact of admitting the ship's speed to vary in space. Let us summarize after the slight refinement of the previous investigation which became the point of reference and the motivation for our study. We thus obtain the following 
\begin{proposition}
\label{thm_kropina}
A metric $\tilde{F}$ is of $\tilde{U}$-Kropina type if and only if it solves the generalized Zermelo \linebreak navigation problem on a Riemannian manifold $(M, h)$, with varying in space ship's speed \linebreak $0<|u(x)|_h\leq1$ in the presence of a strong wind $\tilde{W}(x)$ which satisfies $|\tilde{W}|_h=|u|_h$.
\end{proposition}
Remark that we exclude here $|W|_h=0$ unlike the generalized Randers case with a spatial function $|u(x)|_h$ in the presence of a weak perturbation $W_R$, i.e. $0\leq h(W_R, W_R)<h(u, u)$, where the solutions to the problem are  then determined by the Riemannian metric conformal to $h$. From \eqref{conformality} it yields that the resulting Kropina geodesics of $F$ and $\tilde{F}$ with $|u|_h=const. $ trace the same curves, however  the speeds differ and therefore the times of travel between given points change. This refers to the particular situation in the generalized Randers case, i.e. with $W_R=0$ and $h(u, u)=const.$ Then, however,  Randers metric is reduced to the corrresponding background Riemannian metric $h$ up to scaling; for more details see the study in \cite{kopi6}. Such a case also corresponds to a pair of conformal homothetic Finsler metrics, that is a special case of weakly conformally equivalent Finsler metrics considered in \cite{rafie}. 
Going further, a glance at the new metric \eqref{kropinka} and \eqref{conformality} leads to 
\begin{lemma}
\label{lemat_kropina}
With arbitrary navigation data 
$(h, |u(x)|_h,\tilde{W}(x))$ 
a transit time of existing, 
 nonzero ($u\neq -\tilde{W}$) solution to the generalized Zermelo navigation problem on Riemannian manifolds in the presence of a strong wind satisfying $|\tilde{W}|_h=|u|_h\neq1$,  
 is greater than a transit time of the corresponding solution to the original Zermelo navigation problem. 
\end{lemma}
\begin{proof}
For any piecewise $C^{\infty}$ curve $\ell$ in $M$, the $\tilde{F}$- length of $\ell$ denoted by $\mathcal{L}_{\tilde{F}}(\ell)$ is equal to the time for which the object travels along it, i.e. $T= \int\limits_{0}^{T}\tilde{F}(\dot{\ell}(t))dt = \mathcal{L}_{\tilde{F}}(\ell).$ Let $\gamma$, $\tilde{\gamma}$ be $F$- and $\tilde{F}$-geodesic, respectively, where $\tilde{F}$ is given by \eqref{kropinka}. For any nonzero $ y \in T_x M$ $F(y),\tilde{F}(y)>0$. The function $|u(x)|_h$ is variable in space or constant with $0<h(u, u)\leq1$. Since $u\neq -\tilde{W}$ the resultant speed $v>0$. The equality of the lengths $\mathcal{L}_{\tilde{F}}$ and $\mathcal{L}_{F}$ holds if and only if $|u|_h=1=const.$, then $\tilde{F}(y):=F(y)$. Otherwise, by \eqref{conformality} $\tilde{F}(y) > F(y)$ for any scenario obtained from the triples $(h, \tilde{W}(x), |u(x)|_h)$, thus for any spatial function $|u(x)|_h$ or, equivalently, $|\tilde{W}(x)|_h$, where $|\tilde{W}|_h=|u|_h$. Note that $\mathcal{L}_{\tilde{F}}(\tilde{\gamma})\geq\mathcal{L}_{F}(\tilde{\gamma})$ and $\mathcal{L}_{F}(\tilde{\gamma})\geq\mathcal{L}_{F}(\gamma)$ as the geodesic minimizes the length. From transitivity we are thus led to the inequality $\mathcal{L}_{\tilde{F}}(\tilde{\gamma})\geq\mathcal{L}_{F}(\gamma)$. 
\end{proof}

Remark that the presence of $|u(x)|_h$ in the above expression of navigation data may actually be inessential. If perturbing vector field is a priori fixed then it can be removed, since given $\tilde{W}$ determines $|u|_h$ by $|\tilde{W}|_h$. Nevertheless, we let it to emphasise its new role in the considered approach to the problem inasmuch as we admit $|u(x)|_h$ to be set initially, without being determined by $\tilde{W}$. Next, let us observe that unlike the Randers case, where the entire space $(M, h)$ can be covered with the time-minimal paths, not all the points $x\in M$ are now available for navigating ship any more as the wind is of stronger force. Therefore, one needs to consider the existence of solutions to posed Zermelo problems. 
Also, from 
the above it yields the following, somewhat contrariwise formulated, corollary.  
\begin{corollary}
Let a space-dependent ship's speed $|u(x)|_h$ vary 
  along the existing solution to the generalized Zermelo navigation problem on a Riemannian manifold $(M, h)$ 
 in the presence of a strong perturbation $\tilde{W}\neq -u$ with $|u|_h=|\tilde{W}|_h$. Then the stronger perturbation acts on a ship the shorter travel time is. 

\end{corollary}

\noindent
It implies that a ship reaches her destination in the absolutely shortest time when the strongest perturbation blows, what one may find self-contradictory. Indeed, if $h(\tilde{W}, \tilde{W})=1$ the passage will be time-minimal amidst all the 
possible combinations of navigation data. That is the original formulation of the navigation problem brings an optimal solution in comparison to others modified by a new variable $|u(x)|_h$. Obviously, increasing a wind's force causes that a ship's speed is also boosted to hold the constant ratio $\frac{|\tilde{W}|_h}{|u|_h}=1$ as assumed in the problem. 

Now, let us take a look at different straightforward scenario in the presence of the same strong and varying in magnitude wind $\tilde{W}(x)$, with two Riemannian seas $(M, h)$ and $(M, \hat{h})$ which are determined by the conformal background Riemannian metrics $\hat{h}$ and $h$, where $\hat{h}=\frac{1}{|u|_h^2}h$. Thus, $\hat{h}(\tilde{W},\tilde{W})=1=\hat{h}(u, u)$ since, by assumption, $h(\tilde{W}, \tilde{W})=h(u, u) $. Therefore, this gives 
\begin{equation}
\hat{F}(x, y)=
\frac{\hat{h}(y,y)}{2\hat{h}(y,\tilde{W}(x))}=
\frac{|y|_h^2}{2h(y, \tilde{W}(x))}.
\end{equation}

\noindent
Hence, recalling \eqref{kropinka} it yields the equality of the above Kropina metrics, namely  $\hat{F}=\tilde{F}$. We are thus led to 
 
\begin{corollary}
The time-minimal paths of the background conformal Riemannian metrics $\hat{h}$ and $h$, where $\hat{h}=|u(x)|^{-2}_hh$, perturbed by a strong varying 
in space wind $\tilde{W}(x)$ which satisfies $|\tilde{W}|_h=|u|_h$, are represented by the same Kropina geodesics. 
\end{corollary}

\noindent
In what follows, we present the flow of Kropina geodesics in the generalized approach to the Zermelo navigation problem, with the presence of a strong wind including the influence of a spatial function $|u(x)|_h$. We also compare it to the corresponding solution obtained from the original expression of the problem on the same Riemannian sea $(M, h)$. 


\section{Example}

With the topological restrictions in mind which refer to the existence of globally defined Kropina metrics on $M$, admitting however $\mathbb{S}^{2m-1}$ or $\mathbb{E}^n$, in what follows we present the example with the Euclidean background, namely $\mathbb{E}^2$. Considering dimension two we denote the position coordinates $(x^1, x^2)$ by $(x, y)$ and expand arbitrary tangent vectors $y^1\frac{\partial}{\partial x^1}+y^2\frac{\partial}{\partial x^2}$ at  $(x^1, x^2)$ as $(x,y;u,v)$ or $u\frac{\partial}{\partial x} + v\frac{\partial}{\partial y}$. We also express a ship's speed $|u|_h$ as $|U|$ and the resultant speed $|v|_h$ as $|V|$. Thus, from \eqref{kropinka} we obtain 


\begin{equation}
\tilde{F}(x,y; u, v)=\frac{
h_{11}u^2+2h_{12}uv+h_{22}v^2}{2(h_{11}u\tilde{W}^1+h_{12}u\tilde{W}^2+h_{21}v\tilde{W}^1+h_{22}v\tilde{W}^2)}.
\end{equation}
\noindent
After having set $M:=\mathbb{R}^2$ we get 
\begin{equation}
\tilde{F}(x, y; u,v)=\frac{u^2+v^2}{2(u\tilde{W}^1+v\tilde{W}^2)}
\end{equation}
where simply $|(u, v)|_h=\sqrt{u^2+v^2}$. Without loss of generality let us consider the strong unit wind $W$ represented by 
\begin{equation}
W(x, y)=\cos(x+y)\frac{\partial}{\partial x}+\sin(x+y)\frac{\partial}{\partial y}. 
\label{poleW}
\end{equation}
Hence, $|W(x, y)|_h=\sqrt{(W^1(x, y))^2+(W^2(x, y))^2}=\sqrt{\cos^2(x+y)+\sin^2(x+y)}=1 \  \forall \  (x, y)\in \mathbb{R}^2$. 
\noindent
Consequently, for the applied perturbation \eqref{poleW} the form of the resulting metric in the original expression yields 
\begin{equation}
F(x ,y; u,v)=\frac{u^2+v^2}{2(uW^1+vW^2)}=\frac{u^2+v^2}{2[ u \cos (x+y)+ v \sin (x+y)]}.
\label{kropinka_s2_original}
\end{equation}

\noindent
Let $|u(x)|_h$ be smooth and positive determined by a Gaussian function which is expressed in the general form $f(x)=\bar{a}e^{-\frac{(x-\bar{b})^2}{2\bar{c}^2}}$, where $\bar{a}, \bar{b}, \bar{c}$ are the real constants. For instance, let  
$|U(x,y)|=\frac{2}{3} \exp \left(-\frac{y^2 \sin^2 (x+y)}{\pi }\right)+\frac{1}{3}$. Hence, $ \forall \  (x, y)\in \mathbb{R}^2  \quad |U(x,y)|\in \left(\frac{1}{3}, 1\right]\subset \left(0, 1\right]$. 
The new non-unit wind $\tilde{W}$ blowing on the Euclidean sea yields 
\begin{equation}
\tilde{W}(x, y)=\frac{1}{3}\left(2 e^{-\frac{1}{\pi}y^2 \sin^2 (x+y)}+1\right)\cos(x+y)\frac{\partial}{\partial x}+\frac{1}{3}\left(2e^{ -\frac{1}{\pi}y^2 \sin^2 (x+y)}+1\right)\sin(x+y)\frac{\partial}{\partial y}.
\label{poleWu}
\end{equation}
Let us note that $|W(x, y)|=|U(x, y)|^{-1}|\tilde{W}(x, y)|$. The contour plot and the stream density plot taking the scalar field to be the norm of the perturbation $\tilde{W}$ are presented in Figure \ref{contour_stream}.
\begin{figure}[h!]
        \centering
~\includegraphics[width=0.46\textwidth]{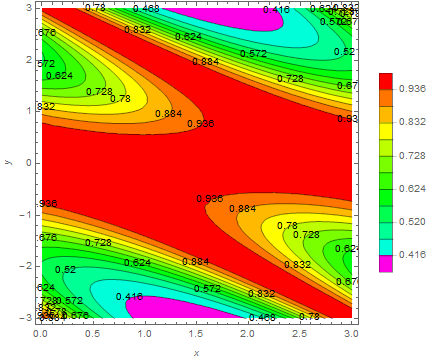} \qquad
~\includegraphics[width=0.44\textwidth]{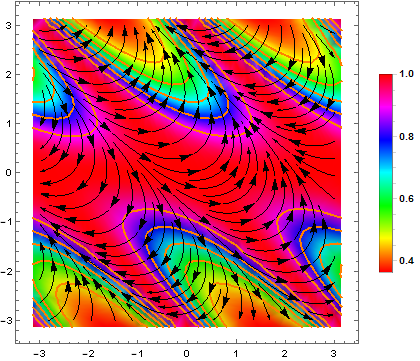} 

        \caption{The contour plot (on the left) and the stream density plot (on the right) of the perturbation $\tilde{W}$ given by \eqref{poleWu}.} 
\label{contour_stream}
\end{figure}
By assumption, $|\tilde{W}(x, y)|_h=\frac{2}{3} \exp \left[-\frac{1}{\pi }y^2 \sin^2 (x+y)\right]+\frac{1}{3}$. For example, with $\varphi_0=0$ a ship commences the voyage starting from the origin with a wind, i.e. $U=\tilde{W}$ and $(\dot{x}, \dot{y})=(2, 0)$ at the maximal resulting speed which equals $|V|:=|U|+|\tilde{W}|=2$ since then $V:=2U$, where $\varphi=\varphi(t)$ is the angle measured counterclockwise which the vector of the relative velocity $U$ forms with $x$-axis. Recall that the function \eqref{kropinka2} is not defined on all $TM$, but only on a domain $\{(x, y; u, v)\in TM: \tilde{\beta}>0\}$. Thus, we have $\varphi_0\in[0, 2\pi) \setminus \{\pi\}$. Clearly, with $\varphi_0=\pi$ , where $\tilde{W}(0, 0)=(1, 0)$, one gets  $(\dot{x}, \dot{y})=(0, 0)$. This means that though a ship proceeds ceaselessly through the water ($|U|>0$), it is stopped over ground, i.e. the resulting speed $|V|=0$. Such a scenario does not occur in the case of Randers metric including the generalized version of the problem in the presence of a weak wind where $|W_R|_h<|u|_h$, cf. \cite{kopi6}.    
For the perturbation \eqref{poleWu} we obtain the form of the resulting metric as follows  
\begin{equation}
\tilde{F}(x ,y; u,v)=\frac{3 \left(u^2+v^2\right) \exp\left(\frac{y^2 \sin ^2(x+y)}{\pi }\right)}{2 \left[\exp\left(\frac{y^2 \sin ^2(x+y)}{\pi }\right)+2\right] \left[u \cos (x+y)+v \sin (x+y)\right]}.
\label{kropinka_s2_generalized}
\end{equation}

\noindent
After having computed the spray coefficients, we obtain the Kropina $F$-geodesic equations for the metric \eqref{kropinka_s2_original} in the original setting. The result is 
\begin{equation}
 \left\{\begin{array}{l l}
		\ddot{x}+\frac{1}{2 \left(\dot{x}^2+\dot{y}^2\right)}\left[\left(4 \dot{x}^3 \dot{y}+4 \dot{x}^2 \dot{y}^2-\dot{x}^4+\dot{y}^4\right) \sin ^2(x+y)+\frac{1}{2}\left(\dot{x}^4+\dot{y}^4\right)\sin 2(x+y)\right.\\
\left.+\dot{x} \dot{y} \left(-3 \dot{x} \dot{y}+2 \dot{x}^2-2 \dot{y}^2\right) \sin 2 (x+y)+2 \dot{y}^2 \left(2 \dot{x} \dot{y}-\dot{x}^2+\dot{y}^2\right) \cos ^2(x+y)\right]=0 

\\

\\

		\ddot{y}- \frac{1}{2 \left(\dot{x}^2+\dot{y}^2\right)}\left[\dot{x} \left(2 \dot{x} \left(2 \dot{x} \dot{y}+\dot{x}^2-\dot{y}^2\right) \sin ^2(x+y)+\dot{y} \left(-3 \dot{x} \dot{y}-2 \dot{x}^2+2 \dot{y}^2\right) \sin 2 (x+y)\right)\right.\\
\left.+\left(4 \dot{x}^2 \dot{y}^2+4 \dot{x} \dot{y}^3+\dot{x}^4-\dot{y}^4\right) \cos ^2(x+y)+\frac{1}{2}\left(\dot{x}^4+\dot{y}^4\right)\sin 2(x+y)\right]=0
\end{array}  \right.
\end{equation}
In the generalization the influence of a spatial function $|U|$ or, equivalently, $|\tilde{W}|$ is noticable in the system of the Kropina $\tilde{F}$-geodesic equations corresponding to    \eqref{kropinka_s2_generalized}. This gives

\begin{equation}
 \left\{\begin{array}{l l}
		\ddot{x}+\frac{1}{2 \pi  \left(\dot{x}^2+\dot{y}^2\right) \left(e^{\frac{y^2 \sin ^2(x+y)}{\pi }}+2\right)} \left\{16 y \dot{x} \dot{y}^3 \sin ^4(x+y)\right.\\
-\pi  \left(-4 \dot{x}^3 \dot{y}-4 \dot{x}^2 \dot{y}^2+\dot{x}^4-\dot{y}^4\right) \left(e^{\frac{y^2 \sin ^2(x+y)}{\pi }}+2\right) \sin ^2(x+y)\\
+\sin 2 (x+y) \left[\pi  \left(\left(2 \dot{x}^3 \dot{y} -3 \dot{x}^2 \dot{y}^2-2 \dot{x} \dot{y}^3\right)\left(e^{\frac{y^2 \sin ^2(x+y)}{\pi }}+2\right) +\dot{x}^4+\dot{y}^4\right)\right.\\
\left.-y \left(-4 (y+1) \dot{x}^3 \dot{y}-6 y \dot{x}^2 \dot{y}^2+4 y \dot{x} \dot{y}^3+y \dot{x}^4+y \dot{y}^4\right) \sin 2 (x+y)\right]\\
+8 y^2 \dot{x}^2 \left(2 \dot{x} \dot{y}+\dot{x}^2-\dot{y}^2\right) \sin (x+y) \cos ^3(x+y)\\
+2 \pi  \dot{y}^2 \left(2 \dot{x} \dot{y}-\dot{x}^2+\dot{y}^2\right) \left(e^{\frac{y^2 \sin ^2(x+y)}{\pi }}+2\right) \cos ^2(x+y)\\
+\frac{1}{2}\sin 2(x+y) \left[4 y \left(2 (2 y+3) \dot{x}^2 \dot{y}^2+4 y \dot{x} \dot{y}^3+(y-1) \dot{x}^4-(y+1) \dot{y}^4\right) \sin ^2(x+y)\right.\\
\left.\left.+\pi  \left(\dot{x}^4+\dot{y}^4\right) e^{\frac{y^2 \sin ^2(x+y)}{\pi }}\right]\right\}=0
 
\\

\\

		\ddot{y}- \frac{1}{2 \pi  \left(\dot{x}^2+\dot{y}^2\right) \left(e^{\frac{y^2 \sin ^2(x+y)}{\pi }}+2\right)}\left\{-8 y \dot{y}^2 \left(\dot{y}^2-\dot{x}^2\right) \sin ^4(x+y)\right.\\
+2 \pi  \dot{x}^2 \left(2 \dot{x} \dot{y}+\dot{x}^2-\dot{y}^2\right) \left(e^{\frac{y^2 \sin ^2(x+y)}{\pi }}+2\right) \sin ^2(x+y)\\
+\sin 2 (x+y) \left[\pi  \left(\left(-2 \dot{x}^3 \dot{y} -3 \dot{x}^2 \dot{y}^2+2 \dot{x} \dot{y}^3\right)\left(e^{\frac{y^2 \sin ^2(x+y)}{\pi }}+2\right) +\dot{x}^4+\dot{y}^4\right)\right.\\
\left.+y \left(4 y \dot{x}^3 \dot{y}-2 (3 y+2) \dot{x}^2 \dot{y}^2-4 y \dot{x} \dot{y}^3+(y+1) \dot{x}^4+(y-1) \dot{y}^4\right) \sin 2 (x+y)\right]\\
+4 y^2 \left(-4 \dot{x}^3 \dot{y}-4 \dot{x}^2 \dot{y}^2+\dot{x}^4-\dot{y}^4\right) \sin (x+y) \cos ^3(x+y)\\
+\pi  \left(4 \dot{x}^2 \dot{y}^2+4 \dot{x} \dot{y}^3+\dot{x}^4-\dot{y}^4\right) \left(e^{\frac{y^2 \sin ^2(x+y)}{\pi }}+2\right) \cos ^2(x+y)+\frac{1}{2}\sin 2(x+y)\\
\left.\left[\pi  \left(\dot{x}^4+\dot{y}^4\right) e^{\frac{y^2 \sin ^2(x+y)}{\pi }}-8 y \dot{y} \left(-y \dot{x}^2 \dot{y}+2 (y+1) \dot{x} \dot{y}^2-2 \dot{x}^3+y \dot{y}^3\right) \sin ^2(x+y)\right]\right\}=0.
\end{array}  \right.
\end{equation}

\noindent
The form of the initial conditions including the optimal control $\varphi(t)$ under perturbing vector field reads $x(0)=x_0\in\mathbb{R}$, $y(0)=y_0\in \mathbb{R}$, and for the first derivative
\begin{equation}
\label{ic1_kropinka}
\dot{x}(0)=\tilde{W}^1(x_0, y_0)+|U(x_0, y_0)|\cos\varphi_0=|U(x_0, y_0)|(\cos(x_0+ y_0)+\cos\varphi_0):=1+\cos\varphi_0, 
\end{equation}
\begin{equation}
\label{ic2_kropinka}
\dot{y}(0)=\tilde{W}^2(x_0, y_0)+|U(x_0, y_0)|\sin\varphi_0 =|U(x_0, y_0)|(\sin(x_0+ y_0)+\sin\varphi_0):=\sin\varphi_0.
\end{equation}

\noindent
The last relations can be derived by direct consideration of the planar  equations of motion including the representation of the vector components of ship's velocity and the new background wind. 
When the families of the time-minimal paths coming from the same fixed point $x\in M$ are considered, $\varphi_0$ plays the role of the parameter which rotates the tangent vector of unperturbed Riemannian geodesic. To provide some numerical computations and to generate the graphs we use Mathematica 10.4 from Wolfram Research. The time-efficient paths in both scenarios, that is Kropina $F$- (black) and $\tilde{F}$-geodesics (red) starting from the origin, with the corresponding strong background winds are presented in Figure \ref{znp_curves}. We set the increments $\Delta \varphi_0 = \frac{\pi}{8}$ and $t=10$. The solutions are also compared accordingly in Figure \ref{znp_curves_razem}. 
\begin{figure}[h]
        \centering
~\includegraphics[width=0.4\textwidth]{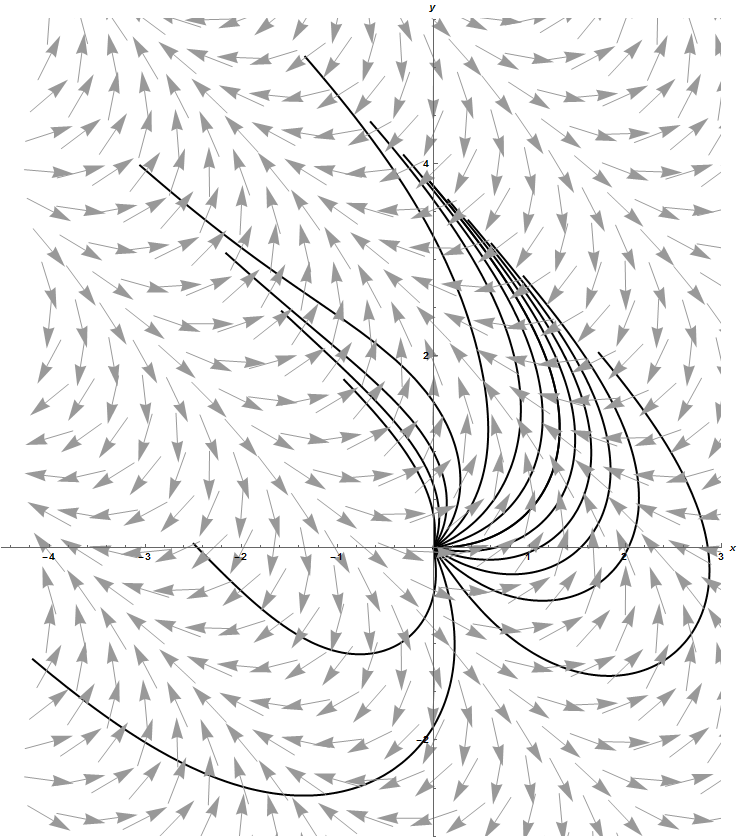} \qquad \qquad
~\includegraphics[width=0.4\textwidth]{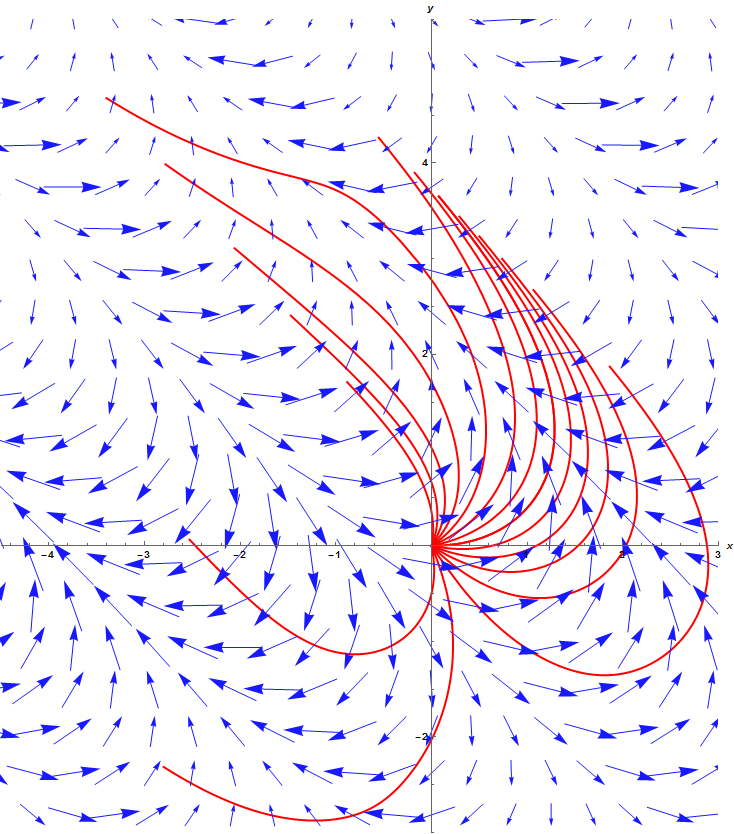} 
        \caption{The Kropina $F$-geodesics (black) with the unit background strong wind $W$ (grey) and $\tilde{F}$-geodesics (red) with the new non-unit background strong wind $\tilde{W}$ (blue), with the increments $\Delta \varphi_0 = \frac{\pi}{8}$; $t=10$.} 
\label{znp_curves}
\end{figure}
\begin{figure}[h!]
        \centering
~\includegraphics[width=0.44\textwidth]{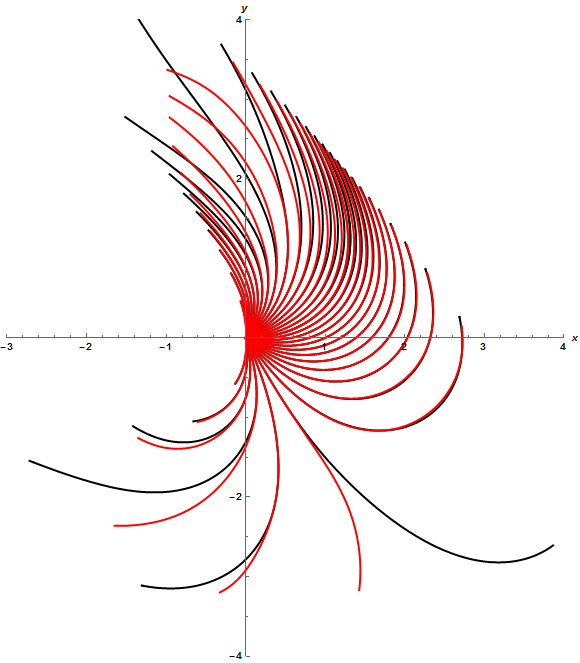}
~\includegraphics[width=0.44\textwidth]{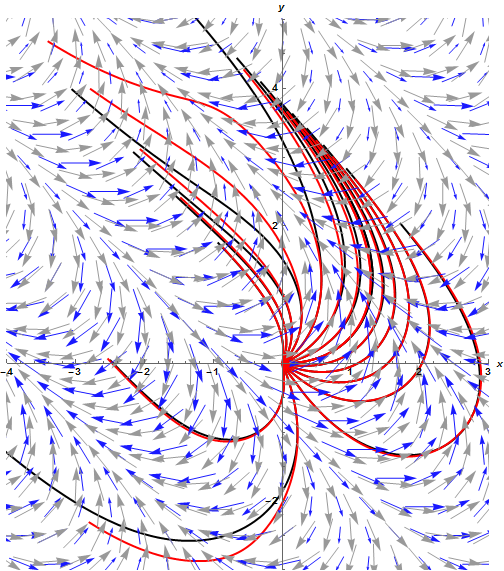} 

        \caption{The Kropina $\tilde{F}$-geodesics (red) starting from the origin compared to the Kropina $F$-geodesics (black) with the increments $\Delta \varphi_0 = \frac{\pi}{18}$, $t=3$ (on the left) and the increments $\Delta \varphi_0 = \frac{\pi}{8}$, $t=10$  (on the right) in the background strong winds $\tilde{W}$ (blue) and $W$ (grey).} 
\label{znp_curves_razem}
\end{figure}
\begin{figure}
        \centering
~\includegraphics[width=0.34\textwidth]{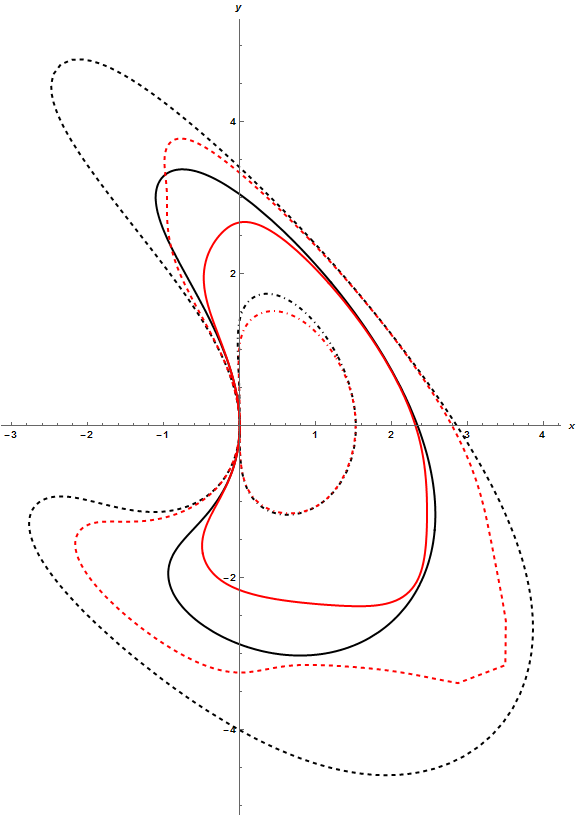} \qquad \qquad
~\includegraphics[width=0.48\textwidth]{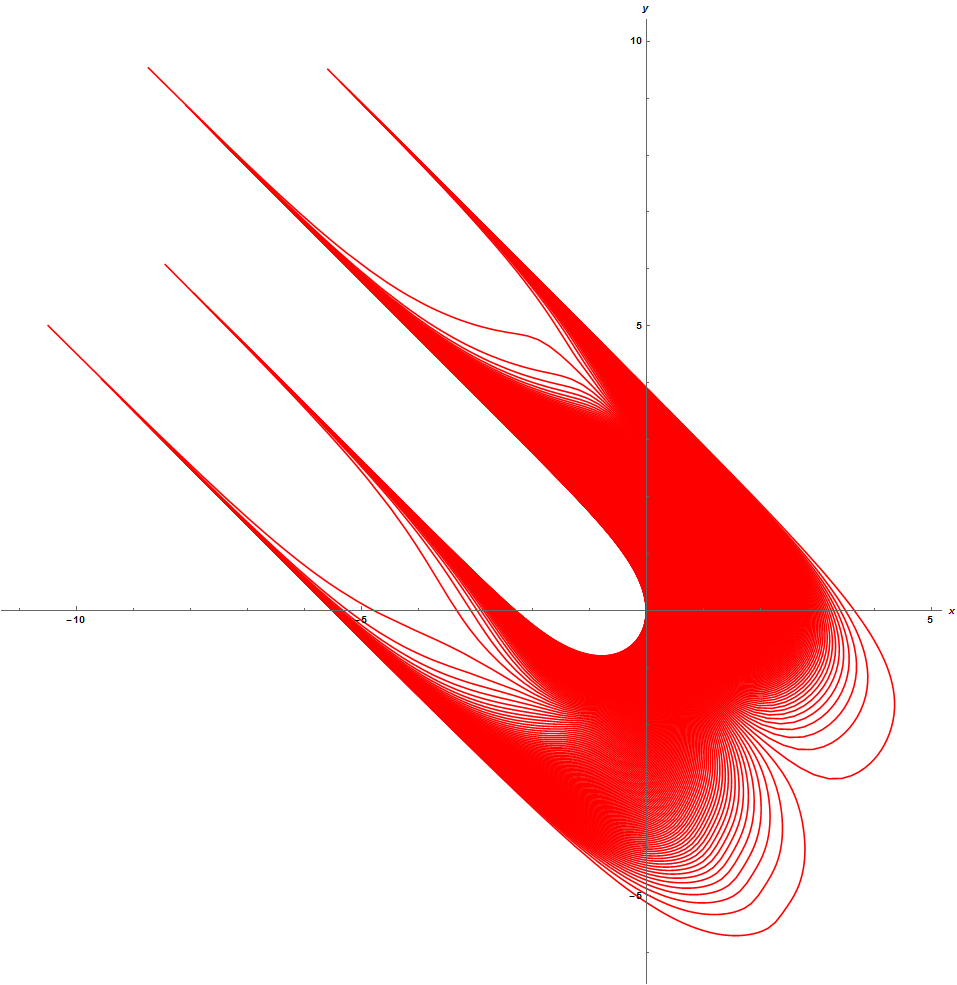}

        \caption{On the left the indicatrices of the Kropina $\tilde{F}$- (red) and $F$-geodesics (black) starting from the origin, with $t=1$ (dot-dashed), $t=2$ (solid), $t=3$ (dashed). On the right the Kropina $\tilde{F}$-geodesics starting from the origin, with the increments $\Delta \varphi_0 = \frac{\pi}{720}$, $t=500$, outlining the area of available points of arrivals.} 
\label{ind}
\end{figure}
The graphical interpretation of Lemma \ref{lemat_kropina} with reference to the example is presented in Figure \ref{ind} where three pairs of $F$- (black) and $\tilde{F}$- indicatrices (red) are compared. It implies that for the corresponding times $t$ the former includes the latter what is the consequence of the influence of applied space-dependent ship's speed. The indicatrix of $\tilde{F}$ is similar to the indicatrix of $F$ with similarity ratio $|u(x)|_h:=|U(x,y)|$. 

Lastly, let us also add that dating back to the formal genesis of the navigation problem in the Hamiltonian formalism, one might investigate the example under consideration with the use of the original navigation formula of E. Zermelo \cite{zermelo, zermelo2, caratheodory} in connection with the results of A. De Mira Fernandes \cite{mira} as we chose the planar Euclidean background. Additionally, the equations of the limit curves which determine the planar area of available points of arrivals as outlined on the right-hand side graph in Figure \ref{ind}, one also might obtain. In this regards, for comparison to the initial research and more details see \S  \ 282 - 287 in \cite{caratheodory}.


\section{Discussion and concluding remarks}

In our study we assumed that the norm $|u(x)|_h$ of a ship's velocity $u$, relative to the surrounding Riemannian sea $(M, h)$ and being a spatial function of $x$, is not necessarily constant, in particular unit, and, what is essential, can be a priori fixed.
In this sense we considered the generalization of the Zermelo navigation with the presence of a strong wind in a purely geometric form. Therefore, we aimed to be in line with the approach to the problem presented in other contributions cited in the introduction and also referred to our previous study for the case of Randers metric, that is the generalization of the navigation problem under a weak wind $|W_R|_h<|u|_h$. Namely, in a starting point we consider the speed $|u(x)|_h$ as a control which complements standard navigation data $(h, \tilde{W})$. Having combined and compared our investigation to the referred meaningful results presented in \cite{kropina} on Kropina metrics making use of the original formulation of the Zermelo navigation with $h(u, u)=1$, we can state that the difference in both approaches refers to the points of view at the problem and the solutions are connected in a simple manner. In what follows we discuss some details and collect the findings. 

In fact, the control $|u(x)|_h$ is strongly limited by the main assumption on the norms' equality, i.e. $|u|_h=|\tilde{W}|_h$ which determines the case. 
One may imagine that in the scenario under consideration there are the "speed zones" referring to the ship's speed through the water or, in other words, the "speed limits" which cover the whole Riemannian sea $(M, h)$. Therefore, captain's duty is to take them into account when preparing the passage plan for the time-efficient voyage by continuous adjusting a ship's engine on the entire route. 
In the second approach the ship's engine telegraph-based plan is executed and the wind force is to be adapted to the fixed ship's passage plan such that it is time-efficient. Though such a scenario is far away from the real marine or air navigation, there are applied optimal control problems when just acting perturbation is fully controllable.    

Proposition \ref{thm_kropina} establishes the direct relation between the Kropina geodesics and the time-minimal paths as the solutions to the navigation problem introducing the space-dependent function $|u(x)|_h$. We see that under the action of a wind $\tilde{W}$ the time-efficient travel path, so the solution to the generalized Zermelo problem, is no longer a background Riemannian $h$-geodesic, but a geodesic of the $\tilde{U}$-Kropina metric $\tilde{F}$. For comparison, let us reflect for a moment on the generalized Randers case (cf. \cite{herdeiro, kopi6}), where one could ask if decreasing a ship's speed $|u|_h$ under fixed weak wind field $W_R$ causes the same effect on the time-minimal path as increasing the wind force with $h(u, u)=1$ and holding the same relation $\frac{|u|_h}{|W_R|_h}$. 
Since $0\leq|W_R|_h<|u|_h<1$, the decrease of the ship's velocity introduces a larger effective wind $\tilde{W}_R^i>W_R^i$. From this point of view the formula for Randers metric in the generalization is then given as in the original setting \cite{colleen_shen}, i.e. $|u|_h=1$, however with $W_R^i$ replaced by a rescaled wind $\tilde{W}_R^i=\frac{1}{|u(x)|_h}W_R^i$. Now, in the presence of stronger perturbation we followed our approach presented in the Randers case what increases the variety of the scenarios and the solutions influenced by the new spatial function $|u(x)|_h$. Remark that, according to Lemma \ref{lemat_kropina}, the corresponding travel times are greater in comparison to the original expression of the problem. Actually, having admitted a priori the space-dependence of $|u(x)|_h$ our presentation  makes a difference in the genesis in comparison to \cite{kropina}. Note that changing $|u(x)|_h$ in the generalized Randers case does not entail the modification of navigation data $(h, W_R)$. Now, with the presence of a strong perturbation, one who sets $|u(x)|_h$ initially may state that it affects the scalar correspondance as the strict condition $|u(x)|_h=|\tilde{W}(x)|_h$ 
is in force. Let us also remark that unlike the Randers case where the entire space can be covered with the time-minimal paths, now not all the destinations are available any more due to the fact that the wind became stronger. In further research one might obtain the general equations or the conditions to be fulfilled for the limit curves which determine the subspace of $M$ including the flows of Kropina geodesics for given navigation data $(h, |u|_h, \tilde{W})$. Therefore, the maps on $M$ including the areas of existing connections in the presence of a strong wind could complement the findings. 

The solutions to the Zermelo problem are represented in the original and the generalized formulation by the same paths up to scaling if $|u|_h=const.$, that is $F$- and $\tilde{F}$-geodesics trace the same curves. Such a case corresponds to a pair of conformal homothetic Finsler metrics, that is a special case of weakly conformally equivalent Finsler metrics considered in \cite{rafie}. The travel times then differ due to the influence of variable $|u(x)|_h$ or, equivalently, $|W(x)|_h$. By Lemma \ref{lemat_kropina} the consequence is the fact that applying any $|u|_h\neq 1$ the passage time will increase in comparison to the original expression which determines the solution of absolutely minimal time. Furthemore, the bijection is established between Kropina spaces represented by pairs  $(\tilde{\alpha}, \tilde{\beta})$ and $(h, W)$ or triples $(h, |u|_h,\tilde{W})$, where $\tilde{W}^i=|u|_hW^i$. Therefore, the generalization with a spatial function $|u(x)|_h$ in the presence of a strong wind corresponds to the original problem with normalized wind, i.e. $W=
\frac{\tilde{W}}{|\tilde{W}|_h}$. This conclusion is in line with the theory on globally defined $U$-Kropina metrics \cite{kropina} where it follows that any Riemannian manifold $(M, h)$ that admits a globally defined nowhere vanishing vector field $W$ can be endowed with a globally defined $U$-Kropina metric. In order to see this, it is remarked that for a Riemannian metric $h$ and a vector field $W$ on $M$ without zeros, it is enough to normalize $\tilde{W}$. 
Then one can construct a $U$-Kropina metric using $h$ and $W$. In fact, the correlated studies coming from the slightly different starting points of view at the navigation problem 
meet. This is caused directly by the main assumption on the norms' equality 
which determines a very special case of the Zermelo navigation problem treated in the paper. 

  




\bigskip

\noindent
\textbf{Acknowledgement} \quad The research was supported by a grant from the Polish National Science Center under research project number 2013/09/N/ST10/02537.  


\bibliographystyle{plain}
\bibliography{pk}

\end{document}